\documentclass[11pt]{article}
\usepackage{enumerate}
\usepackage{amssymb,a4wide,latexsym,makeidx,epsfig,fleqn}
\usepackage{amsthm}
\usepackage{amsmath}
\usepackage{enumerate}
\usepackage{graphicx}
\usepackage{float}
\usepackage{CJK}
\usepackage{indentfirst}
\usepackage{amsmath}
\usepackage{cases}
\allowdisplaybreaks[4]
\newtheorem{theorem}{Theorem}[section]

\newtheorem{lemma}[theorem]{Lemma}

\begin{document}
\textwidth 150mm \textheight 225mm
\title{Distance spectral conditions for $ID$-factor-critical and fractional $[a, b]$-factor of graphs\thanks{Supported by the
National Natural Science Foundation of China (No. 12271439).}}
\author{{Tingyan Ma$^{a,b}$, Ligong Wang$^{a,b,}$\thanks{Corresponding author.
E-mail address: lgwangmath@163.com}}\\
{\small $^a$School of Mathematics and Statistics, Northwestern
Polytechnical University,}\\ {\small  Xi'an, Shaanxi 710129,
P.R. China.}\\
{\small $^b$ Xi'an-Budapest Joint Research Center for Combinatorics, Northwestern
Polytechnical University,}\\
{\small Xi'an, Shaanxi 710129,
P.R. China. }\\
{\small E-mail: matingylw@163.com, lgwangmath@163.com} }
\date{}
\maketitle
\begin{center}
\begin{minipage}{120mm}
\vskip 0.3cm
\begin{center}
{\small {\bf Abstract}}
\end{center}
{\small Let $G=(V(G), E(G))$ be a graph with vertex set $V(G)$ and edge set $E(G)$. A graph is $ID$-factor-critical if for every independent set $I$ of $G$ whose size has the same parity as $|V(G)|$, $G-I$ has a perfect matching. For two positive integers $a$ and $b$ with $a\leq b$, let $h$: $E(G)\rightarrow [0, 1]$ be a function on $E(G)$ satisfying $a\leq\sum _{e\in E_{G}(v_{i})}h(e)\leq b$ for any vertex $v_{i}\in V(G)$. Then the spanning subgraph with edge set $E_{h}$, denoted by $G[E_{h}]$, is called a fractional $[a, b]$-factor of $G$ with indicator function $h$, where $E_{h}=\{e\in E(G)\mid h(e)>0\}$ and $E_{G}(v_{i})=\{e\in E(G)\mid e$ is incident with $v_{i}$ in $G$\}. A graph is defined as a fractional $[a, b]$-deleted graph if for any $e\in E(G)$, $G-e$ contains a fractional $[a, b]$-factor. For any integer $k\geq 1$, a graph has a $k$-factor if it contains a $k$-regular spanning subgraph. In this paper, we firstly give a distance spectral radius condition of $G$ to guarantee that $G$ is $ID$-factor-critical. Furthermore, we provide sufficient conditions in terms of distance spectral radius and distance signless Laplacian spectral radius for a graph to contain a fractional $[a, b]$-factor, fractional $[a, b]$-deleted-factor and $k$-factor.
}
\end{minipage}
\end{center}

\section{Introduction }
\label{sec:ch6-introduction}
Let $G=(V(G), E(G))$ be a simple graph with vertex set $V(G)=\{v_{1}, v_{2}, \ldots, v_{n}\}$ and edge set $E(G)$. We use $|V(G)|=n$ and $|E(G)|=e(G)$ to denote the order and size of $G$.
The complement $\overline{G}$ of $G$ is the graph with $V(\overline{G})=V(G)$ and two distinct vertices in $\overline{G}$ are adjacent if and only if they are non-adjacent in $G$. We use $K_{n}$ and $I_{n}$ to denote a complete graph and an independent set of order $n$. For any nonempty subset $S\subseteq V(G)$, let $G[S]$ be the subgraph of $G$ induced by $S$, and let $G-S$ be the subgraph of $G$ by deleting the vertices in $S$ together with their incident edges. If $S=\{v\}$ we write $G-v$. Similarly, for any nonempty subset $E^{'}\subseteq E(G)$, we denote by $G-E^{'}$ the subgraph of $G$ obtained from $G$ by deleting the edges in $E^{'}$. If $E^{'}=\{e\}$ we write $G-e$. We use $d_{G}(v_{i})$ and $\delta$ to denote the degree of $v_{i}\in V(G)$ and the minimum degree, respectively. The number of isolated vertices of $G$, denoted by $i(G)$, is the cardinality of a set of vertices which are not adjacent to any vertex. Denote by $o(G)$ the number of odd components of $G$, that is, the cardinality of a component with odd number of vertices.

In 1971, Graham and Pollak \cite{Graham1971} introduced the concept of distance matrix of a graph to study the routing of messages or data between computers. The distance matrix $D(G)$ is a real symmetric matrix whose $(i, j)$-entry is the distance $d_{G}(v_{i}, v_{j})$ between $v_{i}$ and $v_{j}$, and the distance is defined as the length of a shortest path from $v_{i}$ to $v_{j}$. The transmission of a vertex $v_{i}$, denoted by $Tr_{G}(v_{i})$, is the sum of distances from $v_{i}$ to all other vertices of $G$, i.e. $Tr_{G}(v_{i})=\sum_{u_{i}\in V(G)}d_{G}(v_{i}, u_{i})$. If $Tr_{G}(v_{i})$ is a constant for each $v_{i}\in V(G)$, then the graph $G$ is transmission regular. The transmission $\sigma(G)$ of a graph is defined to be the sum of distances between every pair of vertices of $G$, that is, $\sigma(G)=\frac{1}{2}\sum_{v_{i}\in V(G)}Tr_{G}(v_{i})$. Let $Tr(G)=diag(Tr_{G}(v_{1}), Tr_{G}(v_{2}), \ldots, Tr_{G}(v_{n}))$ be the diagonal matrix of vertex transmissions in $G$. The distance signless Laplacian matrix of $G$ is defined as $D^{Q}(G)=Tr(G)+D(G)$. The largest eigenvalues of $D(G)$ and $D^{Q}(G)$, denoted by $\lambda_{1}(D(G))$ and $\mu_{1}(D^{Q}(G))$, are called to be distance spectral radius and distance signless Laplacian spectral radius of $G$, respectively. For two vertex-disjoint graphs $G$ and $H$, the disjoint union of $G$ and $H$, denoted by $G+H$, is the graph with vertex set $V(G)\cup V(H)$ and edge set $E(G)\cup E(H)$. In particular, let $tG$ be the disjoint union of $t$ copies of graph $G$. The join graph $G\vee H$ is obtained from $G+H$ by adding all possible edges between $V(G)$ and $V(H)$. The other undefined terms and notions one can refer to \cite{Bondy2008,Brouwer2011,Godsil2001}.

Since Danish mathematicians Petersen first attempted on the study of factors in 1891, the graph factors theory play an important role in graph theory. The definition of a $(g, f)$-factor of $G$ is a spanning subgraph $F$ of $G$ satisfying  $g(v_{i})\leq d_{F}(v_{i})\leq f(v_{i})$ for any vertex $v_{i}$ in $V(G)$, where $g$ and $f$ be two integer-valued functions defined on $V(G)$ such that $0\leq g(v_{i})\leq f(v_{i})$ for each vertex $v_{i}$ in $V(G)$.  A spanning subgraph of a graph is its subgraph whose vertex set is same as the original graph. Let $a$ and $b$ be two positive integers with $a\leq b$. A $(g, f)$-factor is called an $[a, b]$-factor if $g(v_{i})\equiv a$ and $f(v_{i})\equiv b$ for any $v_{i}\in V(G)$. In special, for a positive integer $k$, an $[a, b]$-factor is called a $k$-factor if $a=b=k$. When $k=1$, $1$-factor also called a perfect matching.

Hall and Tutte gave the following influential and fundamental results in the study of factor theory. Many scholars provided Hall-type and Tutte-type characterization on this basis, which play important roles in the study of factors.

\noindent\begin{lemma}\label{le:ch-1} {\normalfont(Hall \cite{Hall 1935})} Let $G$ be a bipartite graph with bipartition $(A, B)$. Then $G$ has 1-factor if and only if $|A|=|B|$ and
$$|N(S)|\geq|S|~for~all~S\subseteq X,$$
where $|N(S)|$ is the set of all neighbours of the vertices in $S$.
\end{lemma}

\noindent\begin{lemma}\label{le1} {\normalfont(Tutte \cite{Tutte 1947})} A graph $G$ has a 1-factor if and only if $o(G-S)\leq|S|$ for every $S\subseteq V(G)$.
\end{lemma}

A graph $G$ is said to be factor-critical if for every vertex $v_{i}$ of $G$, $G-v_{i}$ has a 1-factor. If for every independent set $I$ of $G$ whose size has the same parity as $|V(G)|$, $G-I$ has a perfect matching, then the graph is independent-set-deletable factor-critical, and denoted by $ID$-factor-critical shortly.

Let $h$: $E(G)\rightarrow [0, 1]$ be a function on $E(G)$ satisfying $g(v_{i})\leq\sum _{e\in E_{G}(v_{i})}h(e)\leq f(v_{i})$ for any vertex $v_{i}\in V(G)$. Then the spanning subgraph with edge set $E_{h}$, denoted by $G[E_{h}]$, is called a fractional $(g, f)$-factor of $G$ with indicator function $h$, where $E_{h}=\{e\in E(G)\mid h(e)>0\}$ and $E_{G}(v_{i})=\{e\in E(G)\mid e$ is incident with $v_{i}$ in $G$\}. Corresponding to the concept of $[a, b]$-factor, A fractional $(g, f)$-factor is called a fractional $[a, b]$-factor if $g(v_{i})\equiv a$ and $f(v_{i})\equiv b$ for any $v_{i}\in V(G)$. In particular, a fractional $[a, b]$-factor is called a fractional $k$-factor if $a=b=k$. When $k=1$, a fractional $1$-factor also called a fractional perfect matching.

For fractional 1-factors, Tutte also obtained the following criterion.

\noindent\begin{lemma}\label{le2} {\normalfont(Tutte \cite{Tutte 1953})} Let $G$ be a simple graph. Then $G$ has a fractional 1-factor if and only if $i(G-S)\leq|S|$ for every $S\subseteq V(G)$.
\end{lemma}

Based on the definition of fractional $[a, b]$-factor of a graph, a graph is defined as a fractional $[a, b]$-deleted graph if for any $e\in E(G)$, $G-e$ contains a fractional $[a, b]$-factor.

With the development of spectral graph theory, the study on factors of graphs \cite{Akiyama2011} attracted much attention. In particular, by using Hall-type and Tutte-type characterizations, many authors have devoted to some sufficient conditions for a graph to have factors. Very recently, Fan and Liu \cite{Fan2023} gave bounds for the adjacency spectral radius to guarantee the existence of fractional $[a, b]$-factor and $ID$-factor-critical graphs. For more researches of the relationships between adjacency spectral radius and factors of graphs, we can see \cite{Ao2023,Ao ar,Fan2022,Wei2023,Zhou2023}.

Distance spectral extremal graph theory as a new and interesting study topic have been further studied, we can see the excellent survey \cite{Aouchiche,Shu2021}. In recent years, Zhang and Lin $\cite{Zhang2021}$ considered $1$-factor and distance spectral radius in graphs and bipartite graphs. Moreover, Li and Miao $\cite{Li2022}$ obtained the sufficient condition of odd factors for graphs in terms of distance spectral radius. At the same time, Zhang and Lin $\cite{Lin2021}$ proved distance spectral radius conditions of $1$-factor in bipartite graphs with given minimum degree, and gave the extremal graph attaining the minimum distance spectral radius, which improved the result of $\cite{Zhang2021}$. Later, Miao and Li \cite{Miao2023} studied an upper bound for distance spectral radius in a connected graph to guarantee the existence of a star factor. Zhang and Dam \cite{Zhang2023} discussed the relationship between $k$-critical-factor and distance spectral radius. In this paper, we firstly give a distance spectral radius condition of $G$ to guarantee that $G$ is $ID$-factor-critical. Furthermore, we provide sufficient conditions in terms of distance spectral radius and distance signless Laplacian spectral radius conditions for a graph to contain a fractional $[a, b]$-factor, fractional $[a, b]$-deleted-factor and $k$-factor.

The rest of this paper is organized as follows. In the next section, we provide several Lemmas which will be used in our proofs. In section 3, we use the distance spectral radius to give a sufficient condition for a graph to be $ID$-factor-critical and characterize the extremal graph, we also give the corresponding proof of it. In section 4, we present the sufficient distance spectral radius and distance signless Laplacian spectral radius conditions for fractional $[a, b]$-factor, fractional $[a, b]$-deleted factor and $k$-factor of $G$.

\section{Preliminaries}
\label{sec:Preliminaries}
In the following, we introduce some important lemmas that will be used in this paper.

Let $M$ be the following $n\times n$ matrix
\[
M=\left(\begin{array}{ccccccc}
M_{1,1}&M_{1,2}&\cdots &M_{1,m}\\
M_{2,1}&M_{2,2}&\cdots &M_{2,m}\\
\vdots& \vdots& \ddots & \vdots\\
M_{m,1}&M_{m,2}&\cdots &M_{m,m}\\
\end{array}\right),
\]
whose rows and columns are partitioned into subsets $X_{1}, X_{2},\ldots ,X_{m}$ of $\{1,2,\ldots, n\}$.
The quotient matrix $R(M)$ of the matrix $M$ (with respect to the given partition)
is the $m\times m$ matrix whose entries are the
average row sums of the blocks $M_{i,j}$ of $M$.
The above partition is called equitable
if each block $M_{i,j}$ of $M$ has constant row (and column) sum. Also, we say that the quotient matrix $R(M)$ is equitable if the partition is called equitable.

\begin{lemma}\label{le4} {\normalfont(\cite{Brouwer2011})}
Let $M$ be a real symmetric matrix and $R(M)$ be an equitable quotient matrix of $M$. Then the eigenvalues of $R(M)$ are also eigenvalues of $M$. Furthermore, if $M$ is nonnegative and irreducible. Then
$$\lambda_{1}(D(G))=\lambda_{1}(R(M)).$$
\end{lemma}

\begin{lemma}\label{le5} {\normalfont(\cite{Lin2021})}
Let $e$ be an edge of a graph $G$ such that $G-e$ is connected. Then
$$\lambda_{1}(D(G))<\lambda_{1}(D(G-e)).$$
\end{lemma}

\begin{lemma}\label{le6} {\normalfont(\cite{Lin2021})}
Let $p\geq2$ and $n_{i}\geq1$ for $i=1, 2, \ldots, p$. If $\sum_{i=1}^{p}n_{i}=n-s$ where $s\geq1$, then
$$\lambda_{1}(D(K_{s}+(K_{n_{1}}+K_{n_{2}}+\cdots+K_{n_{p}})))\geq\lambda_{1}(D(K_{s}\vee(K_{n-s-p+1}+I_{p-1}))).$$
\end{lemma}

\noindent\begin{lemma}\label{le3} {\normalfont(\cite{Fan2023})} A graph $G$ is $ID$-factor-critical if and only if $o(G-I-S)\leq|S|$ for every independent set $I$ such that $|I|$ has the same parity as $|V(G)|$ and every subset $S\subseteq V(G)-I$.
\end{lemma}
\section{The distance spectral radius conditions for $ID$-factor-critical of a graph}

\begin{theorem}\label{th1}
Let $G$ be a connected graph of order $n\geq 7r+4$ with $r\geq1$. If
$$\lambda_{1}(G)\leq\lambda_{1}(I_{r}\vee K_{r}\vee(K_{n-3r-1}+I_{r+1})),$$
then $G$ is $ID$-factor-critical unless $G\cong I_{r}\vee K_{r}\vee(K_{n-3r-1}+I_{r+1})$.
\end{theorem}
\begin{proof}
Assume that a graph $G$ is not $ID$-factor-critical. According to Lemma \ref{le3}, there exists an independent set $I$ whose size has the same parity as $|V(G)|=n$, for a subset $S\subseteq V(G)-I$, such that $o(G-I-S)\geq|S|+1$. We denote $|I|=r$ and $|S|=s$, then $o(G-I-S)\geq s+1$. Note that $n$ and $r$ has the same parity, so $n-r$ is even, and so $o(G-I-S)$ and $s$ have the same parity. Therefore $o(G-I-S)\geq s+2$. We can get that $G$ is a spanning subgraph of
$$I_{r}\vee K_{s}\vee(K_{n_{1}}+K_{n_{2}}+\cdots+K_{n_{s+2}})$$ for all odd integers $n_{1}\geq n_{2}\geq\cdots\geq n_{s+2}>0$ with $\sum_{i=1}^{s+2}n_{i}=n-r-s$. Then by Lemma \ref{le5}, we have
$$\lambda_{1}(D(G))\geq\lambda_{1}(D(I_{r}\vee K_{s}\vee(K_{n_{1}}+K_{n_{2}}+\cdots+K_{n_{s+2}}))),$$
where equality holds if and only if $G=I_{r}\vee K_{s}\vee(K_{n_{1}}+K_{n_{2}}+\cdots+K_{n_{s+2}})$. We denote
$$G^{(s)}=I_{r}\vee K_{s}\vee(K_{n-2s-r-1}+I_{s+1}).$$
By Lemma \ref{le6}, it is easily get that
$$\lambda_{1}(D(G))\geq\lambda_{1}(D(G^{(s)})),$$
where equality holds if and only if $G=G^{(s)}$. Note that $n=s+r+\sum_{i=1}^{s+2}n_{i}\geq2s+r+2\geq 7r+4$. By the condition of Theorem \ref{th1}, we know that $G^{(r)}=I_{r}\vee K_{r}\vee(K_{n-3r-1}+I_{r+1})$ is extremal graph. We use $\lambda_{1}^{*}$ to denote its distance spectral radius of $G^{(r)}$. The main idea of the following is to show that $\lambda_{1}(D(G^{(s)})>\lambda_{1}^{*}$ when $n\geq2s+r+2$ and $s\geq 3r+1$. Let $\textbf{x}$ be the Perron eigenvector of $D(G^{(r)})$, hence $D(G^{(r)})\textbf{x}=\lambda_{1}^{*}\textbf{x}$. It is well-known that $\textbf{x}$ is constant on each part corresponding to an equitable partition. Thus we get
$$\textbf{x}=(\underbrace{a, \ldots, a}_{r}, \underbrace{b, \ldots, b}_{r}, \underbrace{c, \ldots, c}_{n-3r-1}, \underbrace{d, \ldots, d}_{r+1})^{\top},$$
where $a, b, c \in\mathbb{R}^{+}$. In order to compare the appropriate spectral radius, we refine the partition, and write $\textbf{x}$ as follows:
$$\textbf{x}=(\underbrace{a,\ldots,a}_{r}, \underbrace{b,\ldots,b}_{r}, \underbrace{c,\ldots,c}_{s-r}, \underbrace{c,\ldots,c}_{n-2s-r-1}, \underbrace{c,\ldots,c}_{s-r}, \underbrace{d,\ldots,d}_{r+1})^{\top}.$$
Hence, $D(G^{(s)})-D(G^{(r)})$ is partitioned as
$$\bordermatrix{%
& r & r & s-r & n-2s-r-1 & s-r & r+1\cr
\qquad\qquad\quad~~ r & \textbf{0} & \textbf{0} & \textbf{0} & \textbf{0} & \textbf{0} & \textbf{0} \cr
\qquad\qquad\quad~~r & \textbf{0} & \textbf{0} & \textbf{0} & \textbf{0} & \textbf{0} & \textbf{0} \cr
\qquad\qquad s-r & \textbf{0} & \textbf{0} & \textbf{0} & \textbf{0} & \textbf{0} & -\textbf{J} \cr
n-2s-r-1 & \textbf{0} & \textbf{0} & \textbf{0} & \textbf{0} & \textbf{J} & \textbf{0} \cr
\qquad\qquad s-r & \textbf{0} & \textbf{0} & \textbf{0} & \textbf{J} & \textbf{J-I} & \textbf{0} \cr
\qquad\qquad r+1 & \textbf{0} & \textbf{0} & -\textbf{J} & \textbf{0} & \textbf{0} & \textbf{0} },$$
where $\textbf{J}$ is an all-ones matrix of appropriate size and $\textbf{I}$ is an identity matrix. Then we have
\begin{align*}
\lambda_{1}(D(G^{(s)}))-\lambda_{1}^{*}\geq&\textbf{x}^{\top}(D(G^{(s)})-D(G^{(r)}))\textbf{x}\\
=&(s-r)c[(2n-3s-3r-3)c-2(r+1)d].
\end{align*}
Note that $s\geq 3r+1$ and $c>0$, hence it suffices to show that
\begin{align}\label{eqn-1}
(2n-3s-3r-3)c-2(r+1)d>0.
\end{align}
From the vertex partition of $D(G^{(r)})$, we obtain the quotient matrix
$$R(D(G^{(r)}))=\left(\begin{array}{cccc}
  2(r-1) & r & n-3r-1 & r+1 \\
  r & r-1 & n-3r-1 & r+1 \\
  r & r & n-3r-2 & 2(r+1) \\
  r & r & 2(n-3r-1) & 2r
\end{array}\right).$$
By Lemma \ref{le4}, we know that $R(D(G^{(r)}))\textbf{x}=\lambda_{1}^{*}\textbf{x}$, that is,
\begin{subequations}\label{eqn-2}
\begin{numcases}{}%
\lambda_{1}^{*}a=2(r-1)a+rb+(n-3r-1)c+(r+1)d \notag\\
\lambda_{1}^{*}b=ra+(r-1)b+(n-3r-1)c+(r+1)d \notag\\
\lambda_{1}^{*}c=ra+rb+(n-3r-2)c+2(r+1)d \label{eqn-2-1}\\
\lambda_{1}^{*}d=ra+rb+2(n-3r-1)c+2rd \label{eqn-2-2}
\end{numcases}
\end{subequations}
From $(2a)-(2b)$, we can get
\begin{align}
d=\frac{\lambda_{1}^{*}+n-3r}{\lambda_{1}^{*}+2}c.
\end{align}
By substituting $(3)$ and $s\leq\frac{n-r-2}{2}$ into the left side of inequality $(1)$, it suffices to prove that
\begin{align}
\lambda_{1}^{*}>4r+2+\frac{(4r+3)^{2}-1}{n-7r-4}~~(n>7r+4).
\end{align}
Due to the distance spectral radius must larger than the $i$-th row sum of a distance matrix $D(G)$, that is, $\lambda_{1}^{*}>\underset{i}{min}~r_{i}(D(G^{(r)}))=n-1$. It is not difficult to verify that $$n-1>4r+2+\frac{(4r+3)^{2}-1}{n-7r-4}$$
for $n>7r+4$. Therefore, the inequality $(4)$ holds unless $n=7r+4$. In the following, we only need to proof the case $n=7r+4$. If $n=7r+4$ and $n=2s+r+2$, we have $G^{(s)}=I_{r}\vee K_{3r+1}\vee I_{3r+3}$ and $G^{(r)}=I_{r}\vee K_{r}\vee(K_{4r+3}+I_{r+1})$. From the way of the edges connection, by Lemma \ref{le6}, we can directly get that $\lambda_{1}(G^{(s)})>\lambda_{1}^{*}$. Then the proof is finished.
\end{proof}

\section{Sufficient conditions for fractional $[a, b]$-factor of graphs and its complement graphs}
\begin{lemma}\label{le7} {\normalfont(\cite{Fan2023})}
Let $a$ and $b$ be two positive integers with $a\leq b$, and let $G$ be a graph of order $n\geq a+1$ and minimum degree $\delta\geq a$. If
$$e(G)\geq {n-1\choose 2}+\frac{a+1}{2}$$
and $na\equiv 0~(mod~2)$ when $a=b$, then $G$ has a fractional $[a, b]$-factor.
\end{lemma}

\begin{lemma}\label{le8} {\normalfont(\cite{Zhou2023})}
Let $a$ and $b$ be two positive integers with $b\geq \{a, 3\}$, and let $G$ be a graph of order $n$ with $n\geq \{a+2, 7\}$. If $\delta(G)\geq a+1$ and
$$e(G)\geq {n-1\choose 2}+\frac{a+2}{2},$$
then $G$ is a fractional $[a, b]$-deleted graph.
\end{lemma}

\begin{lemma}\label{le9} {\normalfont(\cite{Wei2023})}
Let $G$ be a graph of order $n\geq k+1$ with $kn$ even and minimum degree $\delta\geq k\geq 1$. If
$$e(G)> {n-1\choose 2}+\frac{k+1}{2},$$
then $G$ has a $k$-factor.
\end{lemma}

\noindent\begin{lemma}\label{le10} {\normalfont(\cite{Xing2014})} Let $G$ be a connected graph on $n$ vertices, then
$$\mu_{1}(D^{Q}(G))\geq\frac{4\sigma(G)}{n},$$
with equality holds if and only if $G$ is transmissions regular.
\end{lemma}

Let $W(G)=\sum_{i<j}d_{ij}$ be the Wiener index of a connected graph $G$. Note that $\lambda_{1}(D(G))=max_{\textbf{x}\in \mathbb{R}^{n}}\frac{\textbf{x}^{\top}D(G)\textbf{x}}{\textbf{x}^{\top}\textbf{x}}$. We see that
$$\lambda_{1}(D(G))=max_{\textbf{x}\in \mathbb{R}^{n}}\frac{\textbf{x}^{\top}D(G)\textbf{x}}{\textbf{x}^{\top}\textbf{x}}\geq\frac{\textbf{1}^{\top}D(G)\textbf{1}}{\textbf{1}^{\top}\textbf{1}}=\frac{2W(G)}{n},$$
where $\textbf{1}=(1, 1, \ldots, 1)^{\top}$.

Here, we can easily find that $W(G)=\sigma(G)$.
\noindent\begin{theorem}\label{th2} Let $a$ and $b$ be two positive integers with $a\leq b$, and let $G$ be a graph of order $n\geq a+1$ and minimum degree $\delta\geq a$. If one of the following holds
\begin{itemize}
  \item [\rm (i)] $\lambda_{1}(D(G))\leq n+1-\frac{a+3}{n},$\qquad
  \item [\rm (ii)] $\mu_{1}(D^{Q}(G))\leq 2n+2-\frac{2a+6}{n},$
  \item [\rm (iii)] $\lambda_{1}(D(\overline{G}))\leq 2n-4+\frac{a+3}{n},$
  \item [\rm (iv)] $\mu_{1}(D^{Q}(\overline{G}))\leq 4n-8+\frac{2a+6}{n},$
\end{itemize}
then $G$ has a fractional $[a, b]$-factor.
\end{theorem}

\begin{proof}

(i) We have known that $\lambda_{1}(D(G))\geq\frac{2W(G)}{n}\geq\frac{2\big(e(G)+2\big(\tbinom{n}{2}-e(G)\big)\big)}{n}=2n-2-\frac{2e(G)}{n}$. By condition (i) of Theorem \ref{th2}, we have
$$2n-2-\frac{2e(G)}{n}\leq\lambda_{1}(D(G))\leq n+1-\frac{a+3}{n}.$$
It is easy to check that $e(G)\geq{n-1\choose 2}+\frac{a+1}{2}$. By Lemma \ref{le7}, then $G$ has a fractional $[a, b]$-factor.

(ii) According to the definition of transmission of a vertex $v_{i}$, we have
$$Tr_{G}(v_{i})\geq d(v_{i})\cdot 1+(n-1-d(v_{i}))\cdot 2=2(n-1)-d(v_{i}),$$
with equality holds if and only if the maximum distance between $v_{i}$ and other vertices in $G$ is at most 2. So
$$\sigma(G)=\frac{1}{2}\sum_{v_{i}\in V(G)}Tr_{G}(v_{i})\geq \frac{1}{2}\sum_{v_{i}\in V(G)}[2(n-1)-d(v_{i})]=n(n-1)-e(G),$$
with equality holds if and only if the maximum distance between $v_{i}$ and other vertices in $G$ is at most 2. So
$$\mu_{1}(D^{Q}(G))\geq \frac{4\sigma(G)}{n}\geq 4(n-1)-\frac{4e(G)}{n}.$$
Then by the condition of Theorem \ref{th2}, we have
$$4(n-1)-\frac{4e(G)}{n}\leq \mu_{1}(D^{Q}(G))\leq 2n+2-\frac{2a+6}{n}.$$
It is easy to check that $e(G)\geq{n-1\choose 2}+\frac{a+1}{2}$. Then by Lemma \ref{le7}, we can directly get $G$ has a fractional $[a, b]$-factor.

(iii) According to
$$2n-2-\frac{2e(\overline{G})}{n}\leq\lambda_{1}(D(\overline{G}))\leq2n-4+\frac{a+3}{n}.$$
We can easily get that
$$e(\overline{G})\leq\frac{2n-a-3}{2}.$$
Hence, $e(G)\geq{n\choose 2}-e(\overline{G})\geq{n-1\choose 2}+\frac{a+1}{2}$, and hence, $G$ has a fractional $[a, b]$-factor.

(iv) Because
\begin{align*}
\sigma(\overline{G})&=\frac{1}{2}\sum_{v_{i}\in V(\overline{G})}Tr_{\overline{G}}(v_{i})\\
&\geq \frac{1}{2}\sum_{v_{i}\in V(G)}[(n-1-d(v_{i}))+2d(v_{i})]\\
&=\frac{1}{2}\sum_{v_{i}\in V(G)}[(n-1+d(v_{i})]\\
&=\frac{1}{2}n(n-1)+e(G).
\end{align*}
By Lemma \ref{le10}, we get
$$\mu_{1}(D^{Q}(\overline{G}))\geq\frac{4\sigma(\overline{G})}{n}\geq2(n-1)+\frac{4e(G)}{n}.$$
Combining with the condition (iii) of Theorem \ref{th2}, we get
$$2(n-1)+\frac{4e(G)}{n}\leq \mu_{1}(D^{Q}(\overline{G}))\leq 4n-8+\frac{2a+6}{n}.$$
We can directly get that $e(G)\geq{n-1\choose 2}+\frac{a+1}{2}$. Then $G$ has a fractional $[a, b]$-factor.
\end{proof}
Similarly, by Lemma \ref{le8} and Lemma \ref{le9}, we can get the results as follows.

\noindent\begin{theorem}\label{th3} Let $a$ and $b$ be two positive integers with $b\geq \{a, 3\}$, and let $G$ be a graph of order $n$ with $n\geq \{a+2, 7\}$. If $\delta(G)\geq a+1$ and one of the following holds
\begin{enumerate}
  \item [\rm (i)]$ \lambda_{1}(D(G))\leq n+1-\frac{a+4}{n},$
  \item [\rm (ii)]$ \mu_{1}(D^{Q}(G))\leq 2n+2-\frac{2a+8}{n},$
  \item [\rm (iii)]$ \lambda_{1}(D(\overline{G}))\leq 2n-4+\frac{a+4}{n},$
  \item [\rm (iv)]$ \mu_{1}(D^{Q}(\overline{G}))\leq 4n-8+\frac{2a+8}{n},$
\end{enumerate}
then $G$ has a fractional $[a, b]$-deleted factor.
\end{theorem}

\noindent\begin{theorem}\label{th4} Let $G$ be a graph of order $n\geq k+1$ with $kn$ even and minimum degree $\delta\geq k\geq 1$. If one of the following holds
\begin{enumerate}
  \item [\rm (i)]$ \lambda_{1}(D(G))< n+1-\frac{k+3}{n},$
  \item [\rm (ii)]$ \mu_{1}(D^{Q}(G))< 2n+2-\frac{2k+6}{n},$
  \item [\rm (iii)]$ \lambda_{1}(D(\overline{G}))< 2n-4+\frac{k+3}{n},$
  \item [\rm (iv)]$ \mu_{1}(D^{Q}(\overline{G}))< 4n-8+\frac{2k+6}{n},$
\end{enumerate}
then $G$ has a $k$-factor.
\end{theorem}

\section*{Declaration of competing interest}

There is no competing interest.

\section*{Data availability}

No data was used for the research described in the article.

\end{document}